\documentclass{birkjour}
 \newtheorem{thm}{Theorem}[section]
 \newtheorem{cor}[thm]{Corollary}
 
 \newtheorem{problem}[thm]{Problem}
 \newtheorem{prop}[thm]{Proposition}
 \theoremstyle{definition}
 \newtheorem{defn}[thm]{Definition}
 \theoremstyle{remark}
 \newtheorem{rem}[thm]{Remark}
 \newtheorem*{ex}{Example}
 \numberwithin{equation}{section}

\begin{document}

%-------------------------------------------------------------------------
% editorial commands: to be inserted by the editorial office
%
%\firstpage{1} \volume{228} \Copyrightyear{2004} \DOI{003-0001}
%\seriesextraline{This is the Concrete Title of this Book\br H.E. R and S.T.C. W, Eds.}
%\firstpage{1}

\title[Boundary representations and  rectangular hyperrigidity]
 {Boundary representations and  rectangular hyperrigidity}

%----------Author 1
\author[Arunkumar]{Arunkumar C.S.}

\address{%
Kerala School of Mathematics,\\
 Kozhikode - 673 571,\\
  India. }

\email{arunkumarcsmaths9@gmail.com}

%----------Author 2
\author[Shankar]{Shankar P.}
\address{ % 
Indian Statistical Institute,\\ Statistics and Mathematics Unit, 8th Mile, Mysore Road,\\ Bangalore, 560059, India.}

\email{shankarsupy@gmail.com}

%----------Author 3
\author[Vijayarajan]{Vijayarajan A.K.}

\address{%
Kerala School of Mathematics,\\
 Kozhikode - 673 571,\\
  India. }

\email{vijay@ksom.res.in}



\date{January 1, 2004}

\begin{abstract}
We explore  connections between  boundary representations of operator spaces and  those of the associated Paulsen systems. Using the notions of finite representation and separating property which we introduce for operator spaces, the boundary representations for operator spaces are characterized.   We also  introduce weak boundary for operator spaces. Rectangular hyperrigidity for operator spaces introduced here is used to  establish an analogue of Saskin's theorem in the setting of operator spaces in finite dimensions.
\end{abstract}

%%% ----------------------------------------------------------------------
\maketitle
%%% ----------------------------------------------------------------------
%\tableofcontents

\section{Introduction}
Noncommutative approximation and extremal theories initiated by Arveson \cite{Arv11} in the context of operator systems in $C^*$-algebras have seen tremendous growth in the recent past. One of the main results of the classical approximation theory is the famous Korovkin theorem \cite{PPK60}  which concerns convergence of positive linear maps on function algebras. The classical Korovkin theorem is as follows:  for each $n\in \mathbb{N}$, let $\Phi_n:C[0,1]\rightarrow C[0,1]$ be a positive linear map. If $\lim\limits_{n\rightarrow\infty} ||\Phi_{n}(f)-f ||=0 $ for every $f\in \{1,x,x^2\}$, then  $\lim\limits_{n\rightarrow\infty} ||\Phi_{n}(f)-f ||=0 $ for every $f\in C[0,1]$. The set $G=\{1,x,x^2\}$ is called a Korovkin set in $C[0,1]$.

The classical extremal theory concerning Choquet boundary of subalgebras of function algebras is a closely related topic of interest. Given a uniform algebra $\mathcal A \subset C(X)$ and a point $\xi\in X$, the point $\xi$ belong to the \textit{choquet boundary} \cite{BL75} of $\mathcal A$ if the corresponding evaluation functional admits a unique unital completely positive extension to $C(X)$. Arveson \cite{Arv69} introduced the notion of \textit{boundary representation} for an operator system and proposed it as the non-commutative analogue of Choquet boundary of a uniform algebra. Subsequently several other authors carried forward the program initiated by Arveson and the articles  \cite{Arv08,DK15,DM05,MH79} are worth mentioning in this context.

There is a close relation between Korovkin sets and Choquet boundaries in the classical setting as suggested by Saskin \cite{YAS66}. Saskin's theorem states that $G$ is a Korovkin set in $C[0,1]$ if and only if the Choquet boundary of $G$ is $[0,1]$. Arveson \cite{Arv11}  introduced a non-commutative analogue of Korovkin set, which he called a \textit{hyperrigid set}. Arveson studied hyperrigidity in the setting of operator systems in $C^*$-algebras. Along the lines of Saskin's theorem in the classical setting, Arveson \cite{Arv11} formulated  \textit{hyperrigidity conjecture} as follows. For an operator system  $S$ and  the generated $C^*$-algebra $A=C^*(S)$, if every irreducible representation of $A$ is a boundary representation for $S$, then S is hyperrigid.  The  hyperrigidity conjecture inspired several studies in recent years \cite{RCL18,CH18,MK15,NPSV18, SV17}. Arveson showed in \cite{Arv11} that the conjecture is valid whenever $C^*$-algebra has countable spectrum. Davidson and Kennedy \cite{DK16} verified the conjecture in the case when $C^*$-algebra is commutative. Some partial results in other contexts have also appeared in \cite{RC18,CK14}.

The  conjecture by Arveson which came to be known as 'Arveson's conjecture' \cite{Arv69} also concerns boundary representations for operator systems in $C^*$-algebras  which was completely settled by Davidson and Kennedy \cite{DK16}. The conjecture states that every operator system and every unital operator algebra has sufficiently many boundary representations to completely norm it. A natural generalisation of the above mentioned theory in the non-selfadjoint setting can be done in the context of operator spaces in \textit{ternary rings} of operators. The recent work by Fuller, Hartz and Lupini \cite{FHL18} introduced the notion of boundary representations for operator spaces in ternary rings of operators. They established the natural operator space analogue of Arveson's conjecture \cite{Arv69} on boundary representations. Paulsen's 'off-diagonal' technique and associated generalization of Stinespring's dilation theorem \cite{Pau02} for completely contractive maps played a central role in establishing the conjecture, while Arveson's approach yields the existence of unital completely positive non-commutative Choquet boundary, the corresponding adaptation by Fuller, Hartz and Lupini yields the existence of a completely contractive non-commutative Choquet boundary. In the latter case, the analogy with classical theory of function algebra theory is not very satisfactory, since the resulting non-commutative Shilov boundary is not an algebra.  To overcome these difficulties, Clouatre and Ramsey \cite{CR19} developed the completely bounded counterpart to Choquet boundaries. They used this completely bounded non-commutative Choquet boundary to construct a non-commutative Shilov boundary that is still a $C^*$-algebra.

In this paper we show that a boundary representation for an operator space induces a boundary representation for the corresponding Paulsen system and we illustrate this with a couple of examples. We extend the notion of weak boundary representation introduced in \cite{NPSV18} for operator systems to operator spaces and study the relation between  weak boundary representations of an operator space and the corresponding Paulsen system. The notion of finite representation introduced by Arveson in \cite{Arv69} is generalized in the context of operator spaces. In one of the main results of this article, we characterize boundary representations of operator spaces in terms of rectangular operator extreme points, finite representations and separating property of operator spaces. The notion of rectangular hyperrigidity for operator spaces is introduced and a version of Saskin's theorem is established.

This paper is divided into five sections, besides the introduction. In section 2, we gather  necessary background material and results that are required throughout. Section 3 deals with boundary representation for an operator space and the corresponding operator system.  Section 4 introduces weak boundary representation for operator spaces and prove that a weak boundary representation for an operator space  induces a weak boundary representation for the corresponding Paulsen system  and vice versa. In section 5,  \textit{finite representation} for operator spaces and  \textit{separating operator spaces} are introduced. We prove that $\phi$ is a boundary representation for an operator space $X$ if and only if $\phi$ is a rectangular operator extreme point for $X$, $\phi$ is a finite representation for $X$ and $X$ separates $\phi$. In section 6, we introduce the notion of rectangular hyperrigidity for operator spaces in ternary ring of operators(TRO). We prove that if an operator space is rectangular hyperrigid in the TRO  generated by the operator space, then every irreducible representation of the TRO is a boundary representation for the operator space.  A partial answer for the converse of the above result is also provided which is a version of classical Saskin's theorem in this setting.  A relation between  rectangular hyperrigidity of an operator space  and hyperrigidity of the corresponding Paulsen system is also established.

\section{Preliminaries}

A \textit{ternary ring of operators} (TRO) between Hilbert spaces $H$ and $K$ is a norm closed subspace $T$ of $B(H,K)$  such that $xy^*z\in T$ for all $x,y,z\in T$. A TRO $T$ always carries an operator space structure as a closed subspace of $B(H,K)$. A \textit{ triple morphism } between TRO's $T_1$ and $T_2$ is a linear map $\phi: T_1\rightarrow T_2 $ such that $\phi(xy^*z)=\phi(x)\phi(y)^*\phi(z)$ for all $x,y,z \in T_1$.

Let $T$ be a TRO. Then $TT^*:=\overline{lin}\{xy^*:x,y\in T\}$ is called the \textit{left $C^*$-algebra} of $T$ and similarly, $T^*T:=\overline{lin}\{x^*y:x,y\in T\}$ the called the \textit{right $C^*$-algebra} of $T$. From \cite[8.1.17]{BM04} \textit{the linking algebra} $\mathcal{L}(T)$ is defined to be the set of $2\times 2$ matrices:
$$\mathcal{L}(T):=\begin{bmatrix}
TT^* & T\\
T^* & T^*T
\end{bmatrix}. $$
Thus any TRO $T$ can be seen as the 1-2 corner of its linking algebra $\mathcal L(T)$. Note that the linking algebra $\mathcal L(T)$ is a $C^*$-algbara.

A triple morphism $\theta:T\rightarrow B(H,K)$ induces a $*$-homomorphism $\omega :\mathcal{L}(T)\rightarrow B(K\oplus H)$ on the linking algebra \cite[Proposition 2.1]{MH99} such that $$\omega=\begin{bmatrix}
\omega_{1} & \theta\\
\theta^* & \omega_{2}
\end{bmatrix}$$ where, $\omega_{1}:TT^*\rightarrow B(K)$ and $\omega_{2} :T^*T\rightarrow B(H)$ are $*$-representations satisfying $\omega_1(xy^*)=\theta(x)\theta(y)^*$ and $\omega_2(x^*y)=\theta(x)^*\theta(y)$ for all $x,y\in T$. Conversely by \cite[Proposition 3.1.2]{DB11}, if $\omega:\mathcal{L}(T)\rightarrow B(L)$ is a $*$-homomorphism of the $C^*$-algebra $\mathcal{L}(T)$, then there exist Hilbert spaces $H,K$ such that $L=K\oplus H$ and  a triple morphism $\theta:T\rightarrow B(H,K)$ with
$$\omega=\begin{bmatrix}
\omega_{1} & \theta\\
\theta^* & \omega_{2}
\end{bmatrix}$$ where, $\omega_{1}:TT^*\rightarrow B(K)$ and $\omega_{2} :T^*T\rightarrow B(H)$ are $*$-representations satisfying $\omega_1(xy^*)=\theta(x)\theta(y)^*$ and $\omega_2(x^*y)=\theta(x)^*\theta(y)$ for all $x,y\in T$.
Therefore, there is a 1-1 correspondence between the representations of a TRO and the representations of its linking algebra.

The notions of nondegenerate, irreducible and faithful representations of TRO's are natural generalizations of these notions from $C^*$-algebras.
A \textit{ representation } of a TRO T is a triple morphism $ \phi: T\rightarrow B(H,K)$ for some Hilbert spaces $H$ and $K$. A representation $\phi : T\rightarrow B(H,K)$ is \textit{ nondegenerate }if,  whenever $p$, $q$ are projections in $B(H)$ and $B(K)$, respectively, such that $q\phi(x)=\phi(x)p=0$ for every $x\in T$, one has $p=0$ and $q=0$ (equivalently, if $\overline{\phi(T)H}=K$ and $\overline{\phi(T)^* K}=H$). Let $H_1\subseteq H$ and $K_1\subseteq K$ be closed subspaces. The pair $(H_1,K_1)$ is said to be $\phi$-\textit{invariant} if $\phi(T)H_1\subseteq K_1$ and $\phi(T)^*K_1\subseteq H_1$.
A representation $\phi : T\rightarrow B(H,K)$ is \textit{ irreducible } if  whenever $p$, $q$ are projections in $B(H)$ and $B(K)$ respectively such that $q\phi(x)=\phi(x)p$ for every $x\in T$, one has $p=0$ and $q=0$, or $p=1$ and $q=1$ (equivalently, if $(0,0)$ and $(H,K)$ are the only $\phi$-invariant pairs). We call $\phi$  \textit{faithful} if it is injective or equivalently completely isometric. A TRO $T\subset B(H,K)$ is said to act nondegenerately or   irreducibly if the corresponding inclusion representation is nondegenerate or irreducible, respectively.  A representation of a TRO is nondegenerate (irreducible) if and only if the representation of its linking algebra is nondegenerate (irreducible) \cite[Lemma 3.1.4, Lemma 3.1.5]{DB11}. Let $\phi_i:T\rightarrow B(H_i,K_i),~ i=1,2$, be  representations. The  representations  $\phi_1$ and $\phi_2$ are said to be \textit{unitarily equivalent} if there exists unitary operators $u_1:H_1\rightarrow H_2$ and $u_2:K_1\rightarrow K_2$ such that $\phi_1(t)=u_2^*\phi_2(t)u_1$ for all $t\in T$. We refer the reader to \cite{BM04,DB11} for a nice account on TRO's and the representation theory of TRO's.

Given an operator space $X\subset B(H,K)$, we can assign an operator system $\mathcal S(X)\subset B(K\oplus H)$. This operator system is called the \textit{Paulsen system} \cite[Lemma 8.1]{Pau02} of $X$. The Paulsen system of $X$  is defined to be the space of operators $$\left\{ \begin{bmatrix} \lambda I_{K} & x\\ y^* & \mu I_{H} \end{bmatrix}: x,y \in X, \lambda , \mu \in\mathbb{C}\right \}$$ where $I_{H}$ and $I_{K}$ denote the identity operators on $H$ and $K$ respectively.  Any completely contractive map $\phi : X\rightarrow B(\tilde{H},\tilde{K})$ on the operator space $X$ extends canonically  to a unital completely positive  map $\mathcal S(\phi):\mathcal S(X)\rightarrow B(\tilde{K}\oplus \tilde{H})$ defined by $$\mathcal S(\phi)(\begin{bmatrix} \lambda I_{K} & x\\ y^* & \mu I_{H} \end{bmatrix})=\begin{bmatrix} \lambda I_{\tilde{K}} & \phi(x)\\ \phi(y)^* & \mu I_{\tilde{H}} \end{bmatrix}.$$

Let $T$ be the TRO containing $X$ as a generating subspace. Suppose $\mathcal{A}=C^*(\mathcal{S}(X))$ is the $C^*$-algebra generated by $\mathcal{S}(X)$, then
$$\mathcal{A}=\left\lbrace \left[ {\begin{array}{cc}
TT^*+\lambda I_K  & T \\
T^* & T^*T+\mu I_H
\end{array} } \right]    : \lambda, \mu \in \mathbb{C}  \right\rbrace.$$
Observe that, $\mathcal{A}$ is a unitalization of the linking algebra $\mathcal{L}(T)$ of $T$. Thus, there is a 1-1 correspondence between the representations of TRO $T$ and the $*$-representations of the $C^*$-algebra $\mathcal{A}$.

The following definitions and results are due to Fuller, Hartz and Lupini \cite{FHL18}. Let $X$ be an operator space. A \textit{rectangular operator state} is a nondegenerate completely contractive linear map $\phi:X \rightarrow B(H,K)$ such that $||\phi||_{cb}=1$. Rectangular operator state $\psi:X\rightarrow B(\tilde{H},\tilde{K})$ is said to be a \textit{dilation} of $\phi$ if there exist linear isometries $v:H\rightarrow \tilde{K}$ and $w:K\rightarrow \tilde{K}$ such that $w^*\psi(x)v=\phi(x)$ for every $x\in X$.
Let $\psi:X \rightarrow B(\tilde{H},\tilde{K})$ be a dilation of a rectangular operator state $\phi: X \rightarrow B(H,K)$. We can assume that $H\subset \tilde{H}$ and $K\subset \tilde{K}$. Let $p$ be the orthogonal projection from $\tilde{H}$ onto $H$ and let $q$ be the orthogonal projection from $\tilde{K}$ onto $K$. The dilation  $\psi$ is trivial if
$$\psi(x)=q\psi(x)p+(1-q)\psi(x)(1-p) $$
for every $x\in X$. The operator state $\phi$ on an operator space $X$ is \textit{maximal} if it has no nontrivial dilation.

Let $X$ be a subspace of TRO $T$ such that $T$ is generated as a TRO by $X$. The operator state $\phi$ on $X$ has the \textit{unique extension property} if any rectangular operator state $\tilde{\phi}$ of $T$ whose restriction to $X$ coincides with $\phi$ is automatically a triple morphism.  A rectangular operator state $\phi:X\rightarrow B(H,K)$ is a \textit{boundary representation} for $X$ if $\phi$ has unique extension property and the unique extension of $\phi$ to $T$ is an irreducible representation of $T$. Let $\phi:X\rightarrow B(H,K)$ is a rectangular operator state of $X$, and $T$ is a TRO containing $X$ as a generating subspace. Then $\phi$ is maximal if and only if it has unique extension property.

\section{Boundary representations}
 We investigate the possible relation between boundary representations of an operator space $X$ and the boundary representations of the Paulsen system $\mathcal S(X)$. Fuller, Hartz and Lupini \cite[Proposition 1.8]{FHL18} proved that a boundary representation of the Paulsen system induces a boundary representation of the operator space. Here, we establish the converse of \cite[Proposition 1.8]{FHL18}.

\begin{thm}\label{opsbdy}
If a rectangular operators state  $\phi : X\rightarrow B(H,K)$ is a boundary representation for $X$, then $\mathcal{S}(\phi)$ is a boundary representation for $\mathcal{S}(X)$.
\end{thm}

\begin{proof}
Assume that the rectangular operator state  $\phi : X\rightarrow B(H,K)$ is a boundary representation for $X$. Let  $\theta : T\rightarrow B(H,K)$  be an irreducible representation such that $\theta _{|_{X}}=\phi$. Let $\mathcal A$ be the $C^*$-algebra generated by $\mathcal{S}(X)$ inside $B(K\oplus H)$, then we have
$$\mathcal{A}=\left\lbrace \left[ {\begin{array}{cc}
TT^*+\lambda I_K  & T \\
T^* & T^*T+\mu I_H
\end{array} } \right]    : \lambda, \mu \in \mathbb{C}  \right\rbrace$$
and
  $$\omega=\begin{bmatrix} \omega_{1} & \theta \\
\theta^* &\omega_{2} \end{bmatrix}$$ is a unital  representation of $\mathcal A$ on $K\oplus H$ such that $\omega _{|_{\mathcal{S}(X)}}=\mathcal{S}(\phi)$, where $\omega_1$ and $\omega_2$ are the representations corresponding to $\theta$ of the respective $C^*$-algebras.

We claim that $\omega$ is irreducible. Let $P$ be a non zero projection in $B(K\oplus H)$ that commutes with $\omega(\mathcal{A})$. In particular, $P$ commutes with  $$\omega\left( \begin{bmatrix} 1 & 0 \\ 0 & 0 \end{bmatrix} \right)=\begin{bmatrix} I_{K} & 0 \\ 0 & 0 \end{bmatrix} \text{ and }\omega\left( \begin{bmatrix} 0 & 0 \\ 0 & 1 \end{bmatrix} \right)=\begin{bmatrix} 0 & 0 \\ 0 & I_{H} \end{bmatrix}.$$
Therefore, $P=p\oplus q$
 where $p$ is a projection on $K$ and $q$ is a projection on $H$. Thus, $(p\oplus q)\omega (x)=\omega (x)(p\oplus q)$ for every $x\in \mathcal{A}$ which implies  that $p\theta(a)=\theta(a)q $ for every $a\in T$. Since $\theta $ is an irreducible representation of $T$, it follows that $p=I_{K}$ and $q=I_{H}$ and therefore $P=I_{K\oplus H}$. Hence $\omega$ is an irreducible representation.

Now, to prove $\mathcal{S}(\phi)$ is a boundary representation for $\mathcal{S}(X)$, it is enough to prove the following. If $\Phi: \mathcal{A} \rightarrow B(K\oplus H)$ is any unital completely positive map with the property  $\Phi_{|_{\mathcal{S}(X)}}=\mathcal{S}(\phi)$, then $\Phi =\omega$. Let $\Phi$ be a such a map. By Stinespring's dilation theorem $$\Phi(a)=V^*\rho(a)V,~~ a\in \mathcal{A}$$ where $\rho :\mathcal{A}\rightarrow B(L)$ is the minimal Stinespring representation and $V:K\oplus H \rightarrow L$ is an isometry. Thus, $\omega _{|_{\mathcal{S}(X)}}=\Phi_{|_{\mathcal{S}(X)}}=V^*\rho(\cdot)V_{|_{\mathcal{S}(X)}}$. Since $\rho$ is a unital representation on $L$, we can decompose $L=K_{\rho}\oplus H_{\rho}$, where $K_{\rho}$ is the range of the orthogonal projection $\rho\left( \begin{bmatrix} 1 & 0 \\ 0 & 0 \end{bmatrix} \right) $ and  $H_{\rho}$ is the range of the orthogonal projection $\rho\left( \begin{bmatrix} 0 & 0 \\ 0 & 1 \end{bmatrix} \right)$.  With respect to this decomposition we have,  $$\rho =\begin{bmatrix} \sigma_{1} & \eta \\ \eta ^*  & \sigma_{2} \end{bmatrix}$$
where $\eta : T\rightarrow B(H_{\rho},K_{\rho})$ is a triple morphism and $\sigma _{1}, \sigma_{2}$ are unital representations of the respective $C^*$-algebras.

We claim that  $V=\begin{bmatrix} v_{1} & 0 \\ 0 & v_{2} \end{bmatrix}$, for isometries $v_{1}: K \rightarrow K_{\rho}$ and $v_{2} : H \rightarrow H_{\rho}$.

We have
\begin{align*}
V^* \begin{bmatrix} 1 & 0 \\ 0 & 0 \end{bmatrix} V=& ~
V^* \begin{bmatrix} \sigma_{1}(1) & 0 \\ 0 & 0 \end{bmatrix} V= V^*\rho\left( \begin{bmatrix} 1 & 0 \\ 0 & 0 \end{bmatrix} \right)V \\
=& ~ \Phi \begin{bmatrix} 1 & 0 \\ 0 & 0 \end{bmatrix}
=\mathcal{S}(\phi) \begin{bmatrix} 1 & 0 \\ 0 & 0 \end{bmatrix}
= \begin{bmatrix} 1 & 0 \\ 0 & 0 \end{bmatrix}.
\end{align*}
Similarly,
$V^* \begin{bmatrix} 0 & 0 \\ 0 & 1 \end{bmatrix} V= \begin{bmatrix} 0 & 0 \\ 0 & 1 \end{bmatrix}$  also holds. 

Since $V$ is an isometry, we must have $V=\begin{bmatrix} v_{1} & 0 \\ 0 & v_{2} \end{bmatrix}$ for isometries $v_{1}$ and $v_{2}$.

 Using $\mathcal{S}(\theta) _{|_{\mathcal{S}(X)}}=\mathcal{S}(\phi)=\Phi _{|_{X}}=V^*\rho V$, we have  $\theta(x)=v_{1}^*\eta(x)v_{2} ~\forall ~x\in X .$  Our assumption $\theta$ is a boundary representation for $X$ implies that $\theta(t)=v_1^*\eta(t)v_2~\forall~t\in T$. Using \cite[Proposition 1.6]{FHL18}, $\theta$ is maximal which implies that $\eta$ is a trivial dilation. We have
 $$\eta(t)=q\eta(t)p+(1-q)\eta(t)(1-p) $$
 for every $t\in T$, where $p=v_2v_2^*$ and $q=v_1v_1^*$. The above equation implies that $\eta(T)v_2H\subseteq v_1K$ and $\eta(T)^*v_1K\subseteq v_2H$. Using the minimality assumption \cite[Page 142]{FHL18} of $\eta$, we have $K_\rho$ is the closed linear span of $\eta(T)\eta(T)^*v_1K\cup \eta(T)v_2H$ and $H_\rho$ is the closed linear span of $\eta(T)^*\eta(T)v_2H\cup\eta(T)^*v_1K$.
Straightforward verification shows that $\eta(T)\eta(T)^*v_1K\cup \eta(T)v_2H \subseteq v_{1}{K}$ and $\eta(T)^*\eta(T)v_2H\cup\eta(T)^*v_1K \subseteq v_{2}{H}$, therefore $K_\rho =v_{1}{K}$ and $H_\rho = v_{2}H$.
 Thus, $v_1$ and $v_2$ are onto. Since $v_1$ and $v_2$ are onto isometries, they are unitaries. Therefore $V$ is  unitary.

Since $\rho$ is a representation and $V$  is unitary,  the equation $\Phi(a)=V^*\rho(a)V,~~ a\in \mathcal{A}$  implies that $\Phi$ is a representation of $\mathcal{A}$. Since $\omega$ and $\Phi$ are representations on $\mathcal{A}=C^*(\mathcal{S}(X))$ and $\omega_{|_{\mathcal{S}(X)}}=\Phi_{|_{\mathcal{S}(X)}}$, we have $\Phi=\omega$.
 \end{proof}

 We give a couple of examples to illustrate the above theorem.

\begin{ex}
Let $X\subset B(H,K)$  be an operator space such that the TRO $T$ generated by $X$ acts irreducibly and such that $T\cap \mathcal{K}(H,K)\neq \{0\}$. Then the identity representation of $T$ is a boundary representation for $X$ if and only if the identity representation of $C^*(\mathcal{S}(X))$ is a boundary representation for $\mathcal{S}(X)$.
To see this, first  assume that the identity representation of $T$ is a boundary representation for $X$. Then by  rectangular boundary theorem \cite[Theorem 1.17]{FHL18} the quotient map $B(H,K)\rightarrow B(H,K)/ \mathcal{K}(H,K)$ is not completely isometric on $X$. Using the same line of argument in the proof of the converse part of \cite[Theorem 1.17]{FHL18},  we see that the quotient map $B(K\oplus H)\rightarrow B(K\oplus H)/ \mathcal{K}(K\oplus H)$ is not a completely isometry on $\mathcal{S}(X)$. Then by Arveson's boundary theorem \cite[Theorem 2.1.1]{Arv72}, the identity representation of $C^*(\mathcal{S}(X))$ is a boundary representation for $\mathcal{S}(X)$.

Conversely, if the identity representation of $C^*(\mathcal{S}(X))$ is a boundary representation for $\mathcal{S}(X)$, then by \cite[Proposition 1.8]{FHL18}, identity representation of $T$ is a boundary representation for $X$.
\end{ex}

\begin{ex}
Let $R$ be an operator system in $B(H)$ and let $A=C^*(R)$ be the $C^*$-algebra generated by $R$. In particular, $R$ is an operator space and $A=C^*(R)$ is  itself a TRO generated by $R$. We have $C^*(\mathcal{S}(R))= M_{2}(A) $. Using Hopenwasser's result \cite{AH73} , we can conclude that if $\pi$ is a boundary representation  of  $A$ for $R$, then  $\mathcal{S}(\pi)$ is a boundary representation of $C^*(\mathcal{S}(R))$ for $\mathcal{S}(R)$.
\end{ex}

\begin{cor}\label{spacetopalsym}
If a rectangular operators state  $\phi : X\rightarrow B(H,K)$ has  unique extension property for $X$, then $\mathcal{S}(\phi)$ has the unique extension property for $\mathcal{S}(X)$.
\end{cor}

\begin{proof}
The proof follows from the same line of argument in Theorem \ref{opsbdy}, without the irreducibility assumption.
\end{proof}

\begin{prop}
Suppose $\omega:\mathcal{S}(X)\rightarrow B(L_\omega)$ has  unique extension property on the Paulsen system $\mathcal{S}(X)$ associated with $X$. Then one can decompose $L_\omega$ as an orthogonal direct sum $K_\omega\oplus H_\omega$ in such a way that $\omega=\mathcal{S}(\phi)$ for some rectangular operator state $\phi:X\rightarrow B(H_\omega,K_\omega)$ and $\phi$ on $X$ has  unique extension property.
 \end{prop}

\begin{proof}
The proof follows as in \cite[Proposition 1.8]{FHL18}.
\end{proof}

\section{Weak boundary representations}

Recently, Namboodiri, Pramod, Shankar, and Vijayarajan \cite{NPSV18} introduced a notion of weak boundary representation, which is a weaker notion than Arveson's \cite{Arv69}  boundary representation for operator systems. They studied relations of weak boundary representation with quasi hyperrigidity of operator systems in \cite{NPSV18}. Here we introduce the notion of weak boundary representations for operator spaces as follows:

\begin{defn}
Let $X\subset B(H,K)$ be a operator space and  T be a TRO containing  $X$ as a generating subspace. An irreducible triple morphism $\psi : T\rightarrow B(H,K)$ is called a weak boundary representation for $X$ if  $\psi_{|_{X}}$ has a unique rectangular operator state extension of the form $v^*\psi u$ , namely $\psi$ itself,  where $v:H\rightarrow H$ and $u:K\rightarrow K$ are isometries.
\end{defn}

Suppose $X$ is operator system, $H=K$, $v=u$. Then the above notion of weak boundary representation recovers the weak boundary representation for operator systems. We can observe that all the boundary representations are weak boundary representations for operator spaces.

Now, we investigate the relations between  weak boundary representation of an operator space and the weak boundary representation of it's Paulsen system.

\begin{prop} \label{weak}
Suppose $\omega:C^*(\mathcal S(X))\rightarrow B(L_\omega)$ is a weak boundary representation for the Paulsen system $\mathcal S(X)$ associated with $X$. Then one can decompose $L_\omega$ as an orthogonal direct sum $K_\omega \oplus H_\omega$ in such a way that $\omega=\mathcal S(\theta)$ for some weak boundary representation $\theta: T\rightarrow B(H_\omega,K_\omega)$ for $X$.
\end{prop}

\begin{proof}
Since $\omega$ is an irreducible representation of $C^*(\mathcal{S}(X))$ on $L_\omega$, we can decompose $L_\omega=K_\omega\oplus H_\omega$ such that
$$\omega =\begin{bmatrix}
\omega_1 & \theta \\
\theta^* & \omega_2
\end{bmatrix}, $$
where $\theta:T\rightarrow B(H_\omega, K_\omega)$ is an irreducible representation.

Now, we will prove that $\theta$ is weak boundary representation for $X$. Let  $u:H_\omega \rightarrow H_\omega$ and $v:K_\omega \rightarrow K_\omega$ be isometries
such that $v^*\theta (a) u=\theta(a)~\forall~a\in X$. For every $a,b \in X$,
\begin{eqnarray*}
\begin{bmatrix}
 v^* & 0\\
0 & u^*\end{bmatrix} \omega \begin{bmatrix}
 \lambda _{1}I_{K} & a\\
b^* & \lambda _{2}I_{H}\end{bmatrix} \begin{bmatrix}
 v & 0\\
0 & u\end{bmatrix} &=&
\begin{bmatrix}
 v^* & 0\\
0 & u^*\end{bmatrix} \begin{bmatrix}
 \lambda _{1}I_{K} & \theta(a)\\
\theta(b)^* & \lambda _{2}I_{H}\end{bmatrix} \begin{bmatrix}
 v & 0\\
0 & u\end{bmatrix} \\ &=&  \begin{bmatrix}
 \lambda _{1}I_{K} & v^*\theta(a)u\\
(v^*\theta(b)u)^* & \lambda _{2}I_{H}\end{bmatrix}\\
&=&  \begin{bmatrix}
 \lambda _{1}I_{K} & \theta(a)\\
\theta(b)^* & \lambda _{2}I_{H}\end{bmatrix}.
\end{eqnarray*}
Since $\begin{bmatrix}
 v & 0\\
0 & u\end{bmatrix}$ is an isometry on $K_\omega\oplus H_\omega $ and $\omega$ is a weak boundary representation for $\mathcal S(X)$ we have for all $a\in T$,
\begin{eqnarray*}
\begin{bmatrix}
 v^* & 0\\
0 & u^*\end{bmatrix} \omega \left(\begin{bmatrix}
 0 &  a\\
0 & 0\end{bmatrix}\right) \begin{bmatrix}
 v & 0\\
0 & u\end{bmatrix} &=&
\omega \left(\begin{bmatrix}
 0 & a\\
0 & 0\end{bmatrix}\right) \\
\begin{bmatrix}
 v^* & 0\\
0 & u^*\end{bmatrix}  \begin{bmatrix}
 0 &  \theta(a)\\
0 & 0\end{bmatrix} \begin{bmatrix}
 v & 0\\
0 & u\end{bmatrix} &=&
 \begin{bmatrix}
 0 & \theta(a)\\
0 & 0\end{bmatrix} \\
\begin{bmatrix}
 0 &  v^*\theta(a)u\\
0 & 0\end{bmatrix}  &=&
 \begin{bmatrix}
 0 & \theta(a)\\
0 & 0\end{bmatrix}. \\
\end{eqnarray*}
From the last equality $v^*\theta(a)u=\theta(a)$, for every $a\in T$. Thus $\theta$ a is weak boundary representation for $X$.
\end{proof}

\begin{prop}\label{converse weak}
If $\theta:T\rightarrow B(H,K)$ is a weak boundary representation for $X$, then the corresponding representation $\omega$ of $C^*(\mathcal{S}(X))$ on $B(K\oplus H)$ is a weak boundary representation for the Paulsen system $\mathcal{S}(X)$.
\end{prop}

\begin{proof} Arguing as  in the proof of Theorem \ref{opsbdy}, we have $\omega$ is an irreducible representation. Let $V$ be an isometry on $K\oplus H$ such that $V^*\omega V_{|_{\mathcal{S}(X)}}=\omega_{|_{\mathcal{S}(X)}}$. As $$V^*\begin{bmatrix}1& 0 \\ 0 & 0\end{bmatrix}V=\begin{bmatrix}1& 0 \\ 0 & 0\end{bmatrix}$$
and
$$V^*\begin{bmatrix}0& 0 \\ 0 & 1\end{bmatrix}V=\begin{bmatrix}0& 0 \\ 0 & 1\end{bmatrix},$$
we can factorize $V=\begin{bmatrix}v_{1}& 0 \\ 0 & v_{2}\end{bmatrix}$, where $v_{1}$ and $v_{2}$ are isometries on $K$ and $H$ respectively. Also we have $v_{1}^*\theta v_{2{|_X}}=\theta _{{|_X}}$. Our assumption that $\theta $ is a weak boundary representation for $X$ implies that $v_{1}^*\theta(t)v_{2}=\theta(t)$ for all $t\in T$. Thus, we have $q\theta(t)p=\theta(t)$ for all $t\in T$, where $q$ and $p$ are the projections onto range of $v_1$ and range of $v_2$ respectively. Now, for each $t\in T$, $q\theta(t)=q(q\theta(t)p)=q\theta(t)p=(q\theta(t)p)p=\theta(t)p$. Since $\theta$ is irreducible, we have $p=I_H$ and $q=I_K$.
Therefore $v_{1}$ and $v_{2}$ are unitaries. Consequently, $V$ is a unitary. Thus $V^*\omega V$ is a representation of $C^*(\mathcal{S}(X))$. Since $V^*\omega V= \omega$ on $\mathcal{S}(X)$, we must have   that $V^*\omega V= \omega$ on $C^*(\mathcal{S}(X))$.
\end{proof}

\section{Characterisation of  boundary representations for operator spaces}

Arveson \cite{Arv69} introduced the notion of finite representations in the setting of subalgebras of $C^*$-algebras. Namboodiri, Pramod, Shankar and Vijayarajan \cite{NPSV18} explored the relation between finite representations and weak boundary representations in the context of operator systems. Here, we introduce the notion of finite representation  in the setting of operator spaces.

\begin{defn}
Let $X$ be an operator space generating a TRO $T$. Let $\phi : T\rightarrow B(H,K)$ be a representation. We say that $\phi$ is a \textit{finite} representation for $X$  if  for every isometries $u:H\rightarrow H$ and $v:K\rightarrow K$, the condition  $v^*\phi(x)u=\phi(x)$, for all $x\in X$ implies that $u$ and $v$ are unitaries.
\end{defn}

It is clear that, when $X$ is operator system, $H=K$, $v=u$, the above notion of finite representation recovers the Arveson's  notion of finite representation.

\begin{prop}\label{finite}
Let $\omega:C^*(\mathcal{S}(X))\rightarrow B(L)$ be a finite representation for the Paulsen system $\mathcal{S}(X)$ associated with $X$. Then one can decompose $L$ as an orthogonal direct sum $K\oplus H$ in such a way that $\omega=\mathcal{S}(\phi)$ for some finite representation  $\phi:T \rightarrow B(H, K)$ for $X$.
\end{prop}

\begin{proof}
We can get a triple morphism $\phi:T \rightarrow B(H, K)$ as in the proof of \cite[Proposition 1.8]{FHL18}. Now, we will prove that $\phi$ is a finite representation for $X$.

Let $u:H\rightarrow H$ and $v:K\rightarrow K$ be isometries such that $v^*\phi(x)u=\phi(x)~\forall~x\in X$. Then for all $x,y\in X$
\begin{align*}
\begin{bmatrix}v^* & 0 \\ 0 & u^*\end{bmatrix}\omega\left( \begin{bmatrix}\lambda I_K & x \\ y ^* & \mu I_H \end{bmatrix}\right) \begin{bmatrix}v & 0 \\ 0 & u\end{bmatrix}=& ~ \begin{bmatrix}v^* & 0 \\ 0 & u^*\end{bmatrix}\begin{bmatrix}\lambda I_K & \phi(x) \\ \phi (y)^* & \mu I_H \end{bmatrix}\begin{bmatrix}v & 0 \\ 0 & u\end{bmatrix} \\
=&~\begin{bmatrix}\lambda I_K & v^*\phi(x) u \\ u^*\phi(y) ^*v & \mu  I_H \end{bmatrix}\\
=&~\mathcal{S}(\phi)\left( \begin{bmatrix}\lambda I_K & x \\ y ^* & \mu I_H \end{bmatrix}\right).
\end{align*}
Thus, $\begin{bmatrix}v^* & 0 \\ 0 & u^*\end{bmatrix}\omega   \begin{bmatrix}v & 0 \\ 0 & u\end{bmatrix}_{|_{\mathcal{S}(X)}}= \omega _{|_{\mathcal{S}(X)}}$. Since $\omega$ is a finite representation for $\mathcal{S}(X)$, we have $\begin{bmatrix}v & 0 \\ 0 & u\end{bmatrix}$ is a unitary and consequently $u$ and $v$ are unitaries. Hence $\phi$ is a finite representation for $X$.
\end{proof}

\begin{prop}\label{converse finite}
If $\phi: T\rightarrow B(H,K)$ is a finite representation for $X$, then $\omega=\begin{bmatrix} \omega_{1} & \phi \\
\phi^* &\omega_{2}\end{bmatrix}:C^*(\mathcal{S}(X))\rightarrow B(K\oplus H)$ is finite representation for $\mathcal{S}(X)$.
\end{prop}

\begin{proof}Arguing as the in the proof of Theorem \ref{opsbdy}, we have $\omega$ is a representation.
Let $V:K\oplus H\rightarrow K\oplus H $ be an isometry such that $V^*\omega(a) V=\omega(a)$ for all $a\in \mathcal{S}(X)$. Since
$$ V^*\begin{bmatrix}1 & 0 \\ 0 & 0\end{bmatrix}V=\begin{bmatrix}1 & 0 \\ 0 & 0\end{bmatrix}\text{ and }  V^*\begin{bmatrix}0 & 0 \\ 0 & 1\end{bmatrix}V=\begin{bmatrix}0 & 0 \\ 0 & 1\end{bmatrix},$$
 we can decompose $V$ as $\begin{bmatrix}v & 0 \\ 0 & u\end{bmatrix}$, where $v:K\rightarrow K$ and $u:H\rightarrow H$ are isometries. For $a\in \mathcal{S}(X)$, we have

$$V^*\omega(a)V=\begin{bmatrix}v^* & 0 \\ 0 & u^*\end{bmatrix} \begin{bmatrix} \omega_{1} & \phi \\
\phi^* &\omega_{2}\end{bmatrix}(a) \begin{bmatrix}v & 0 \\ 0 & u\end{bmatrix}  =\begin{bmatrix} \omega_{1} & \phi \\
\phi^* &\omega_{2}\end{bmatrix}(a)=\omega(a).$$

Therefore $v^*\phi(x)u=\phi(x)$ for every $x\in X$. Since $\phi$ is a finite representation for $X$, $v$ and $u$ are unitaries, $V$ is  unitary. Hence $\omega$ is a finite representation for $\mathcal{S}(X)$.
\end{proof}

The following theorem shows a relation between finite representations and  weak boundary representations.
\begin{thm}
Let $X$ be an operator space generating a TRO T. Let $\phi$ be an irreducible representation of $T$. Then $\phi$ is a finite representation for $X$ if and only if $\phi$ is a weak boundary representation for $X$.
\end{thm}

\begin{proof}
Suppose $\phi:T\rightarrow B(H,K)$ is an irreducible finite representation for $X$. By  Proposition \ref{converse finite},  $\omega:C^*(\mathcal{S}(X))\rightarrow B(K\oplus H)$ is an irreducible finite representation for the Paulsen system $\mathcal{S}(X)$. Using \cite[Proposition 3.5]{NPSV18}, we get that $\omega$ is a weak boundary representation of $C^*(\mathcal{S}(X))$ for $\mathcal{S}(X)$. Therefore, Proposition \ref{weak} implies that $\phi$ is a weak boundary representation for $X$.

Conversely, suppose $\phi:T\rightarrow B(H,K)$ is a weak boundary representation of $T$ for $X$. Hence by Proposition \ref{converse weak}, we have $\omega:C^*(\mathcal{S}(X))\rightarrow B(K\oplus H)$  is a weak boundary representation for $\mathcal{S}(X)$. Using \cite[Proposition 3.5]{NPSV18}, we get that $\omega$ is an irreducible finite representation for $\mathcal{S}(X)$. Therefore, Proposition \ref{finite} implies that $\phi$ is an irreducible finite representation for $X$.

\end{proof}

Arveson\cite{Arv69} introduced the notion of separating subalgbras to characterize  boundary representations in the context of subalgebras of $C^*$-algebas. Pramod, Shankar and Vijayarajan \cite{PSV17} studied the separating notion in the setting of operator systems and explored the relation with boundary representations. Here, we introduce the notion of separating operator space as follows:

\begin{defn}
Let $X$ be an operator space generating a TRO $T$. Let $\phi : T\rightarrow B(H,K)$ be an irreducible representation. We say that $X$ \textit{separates} $\phi$  if for every irreducible representation $\psi : T\rightarrow B(\tilde{H},\tilde{K}) $ and isometries $u:H\rightarrow \tilde{H}$ and $v:K\rightarrow \tilde{K}$ , $v^*\psi(x)u=\phi(x)$, for all $x\in X$ implies that $\phi$ and $\psi$ are unitarily equivalent.  Also, the operator space $X$ is called a \textit{ separating operator space } if it separates every irreducible representations of $T$.
\end{defn}

It is clear that, when $X$ is operator system, $H=K$ and $v=u$, the notion of separating operator space recovers the notion of separating operator system.

\begin{prop}\label{separates}
Suppose $\omega:C^*(\mathcal{S}(X))\rightarrow B(L)$ is an irreducible representation and $\mathcal{S}(X)$  separates $\omega$. Then one can decompose $L$ as an orthogonal direct sum $K\oplus H$ in such a way that $\omega=\mathcal{S}(\phi)$ for some irreducible representation $\phi:T\rightarrow B(H,K)$ and $X$ separates $\phi$.
\end{prop}

\begin{proof}
Existence and irreduciblity of a triple morphism $\phi:T \rightarrow B(H, K)$ follows from the proof of \cite[Proposition 1.8]{FHL18}. Now we will prove that $X$ separates $\phi$.

Let $\theta: T\rightarrow B(H_{\theta},K_{\theta})$ be an irreducible representation of $T$ such that $v_{1}^*\theta (x)v_{2}=\phi(x)$ for all $x\in X$, where $v_{1}:K\rightarrow K_{\theta}$ and $v_{2}:H\rightarrow H_{\theta}$ are isometries. Let $\rho:C^*(\mathcal{S}(X))\rightarrow B(K_\theta\oplus H_\theta)$ be the irreducible representation of $C^*(\mathcal{S}(X))$ corresponding to $\theta$. Then for the isometry  $V=\begin{bmatrix}
  v_{1} & 0 \\
  0 & v_{2}
\end{bmatrix}$  we have $V^*\rho V_{|_{\mathcal{S}(X)}}=\omega _{|_{\mathcal{S}(X)}}$. Since $\mathcal{S}(X)$ separates $\omega$, there exists a unitary $U:K\oplus H \rightarrow K_{\theta}\oplus H_{\theta}$ such that $U^*\rho(a) U =\omega(a)$ for all $a\in C^*(\mathcal{S}(X))$. Using
$$ U^*\begin{bmatrix}1 & 0 \\ 0 & 0\end{bmatrix}U=\begin{bmatrix}1 & 0 \\ 0 & 0\end{bmatrix}\text{ and }  U^*\begin{bmatrix}0 & 0 \\ 0 & 1\end{bmatrix}U=\begin{bmatrix}0 & 0 \\ 0 & 1\end{bmatrix}.$$
we can  factorize $U$ as $\begin{bmatrix}
  u_{1} & 0 \\
  0 & u_{2}
\end{bmatrix}$. Thus  $u_{1}^*\theta(t) u_{2}=\phi(t)$ for every $t\in T$. Hence $X$ separates $\phi$.
\end{proof}

\begin{prop}\label{converse separates}
Let $X$ be an operator space and $T$ be the TRO containing $X$ and generated by $X$. If $\phi:T\rightarrow B(H,K)$ is an irreducible representation such that $X$ separates $\phi$, then the corresponding representation  $\omega:C^*(\mathcal{S}(X))\rightarrow B(K\oplus H)$ is an irreducible representation such that  $\mathcal{S}(X)$ separates $\omega$.
\end{prop}

\begin{proof}
Let $\phi:T\rightarrow B(H,K)$ be an irreducible representation. Then the corresponding representation $\omega$ of the $C^*$-algebra $C^*(\mathcal{S}(X))$ on $B(K\oplus H)$ can be written as
$$\omega=\begin{bmatrix}
\omega_1 & \phi \\
\phi^* &\omega_2
\end{bmatrix}, $$
where $\omega_{1}$ and $\omega_{2}$ as earlier.

Using the same line of argument as in the proof of Theorem \ref{opsbdy}, we have $\omega$ is an irreducible representation. We will prove that $\mathcal{S}(X)$ separates $\omega$.

Let $\rho : C^*(\mathcal{S}(X))\rightarrow B(L)$  be an irreducible representation such that  $V^*\rho V _{|_{\mathcal{S}(X)}}=\omega  _{|_{\mathcal{S}(X)}}$ for some isometry $V: K\oplus H\rightarrow L$. Since $\rho$ is an irreducible representation of $C^*(\mathcal {S}(X))$, we can decompose $L=K_{\rho}\oplus H_{\rho}$ such that
$$\rho =\begin{bmatrix}
  \rho _{1} & \theta \\
  \theta ^* & \rho _{2}
\end{bmatrix}$$
where $\theta : T\rightarrow B(H_{\rho}, K_{\rho})$ is an irreducible representation of $T$. Also, we have
$$ V =\begin{bmatrix}
  v_{1} & 0 \\
  0 & v_{2}
\end{bmatrix}
$$ where  $v_{1}: K\rightarrow K_{\rho}$ and $v_{2}: H\rightarrow H_{\rho}$ are isometries.  Substituting the expressions of $\rho$ and $V$ in the equation $V^*\rho(\cdot)V=\omega(\cdot)$, we obtain $v_{1}^*\theta (x)v_{2}=\phi(x)$ for all $x\in X$. Our assumption that $X$ separates $\phi$ implies that there exists unitaries $u_{1}$ and $u_{2}$ such that  $u_{1}^*\theta (t)u_{2}=\phi(t)$ for all $t\in T$. Then we have $u_{2}^*\theta ^*(t)u_{1}=\phi^*(t)$ for all $t\in T$. Now, for each $x,y\in T$, $\omega_1(xy^*)=\phi(x) \phi(y)^*=u_1^*\theta(x)u_2 u_2^*\theta^*(t)u_1=u_1^*\theta(x)\theta(y)^*u_1=u_1^*\rho_1(xy^*) u_1$ and similarly $\omega_2(xy^*)=u_2^*\rho_2(xy^*) u_2$. Therefore $\omega_i$ and $\rho_i$ are unitarily equivalent via the unitary $u_i$, $i=1,2$. Thus,
$$\omega =\begin{bmatrix}\omega_1 & \phi \\
\phi^*& \omega_2 \end{bmatrix}=\begin{bmatrix} u_1^*\rho_1 u_1 & u_1^*\theta u_2 \\
 u_2^*\theta^* u_1 & u_2^*\rho_2 u_2 \end{bmatrix}
 =\begin{bmatrix} u_1^*& 0 \\
 0 & u_2^*\end{bmatrix}
 \begin{bmatrix}\rho_1 & \theta \\
 \theta^* & \rho_2 \end{bmatrix}
 \begin{bmatrix} u_1& 0 \\
 0 & u_2 \end{bmatrix}.$$
Hence $\rho$ is uniatrily equivalent to $\omega$ by the unitary $ U =\begin{bmatrix}
  u_{1} & 0 \\
  0 & u_{2}
\end{bmatrix}.$
\end{proof}

Fuller, Hartz and Lupini \cite{FHL18} introduced the notion of rectangular extreme points. Suppose that $X$ is an operator  space, and $\phi: X\rightarrow B(H,K)$ is a completely contractive linear map. A \textit{rectangular operator convex combination} is an expression $\phi=\alpha _{1}^*\phi _{1}\beta _{1}+\alpha _{2}^*\phi _{2}\beta _{2}+\cdots +\alpha _{n}^*\phi _{n}\beta _{n}$, where $\beta _{i}: H \rightarrow H_{i}$ and $\alpha _{i}: K\rightarrow K_{i}$ are linear maps, and $\phi _{i}: X\rightarrow B(H_{i},K_{i})$  are completely contractive linear maps for $i=1,2,\cdots , n$ such that $\alpha _{1}^*\alpha _{1}+\cdots +\alpha _{n}^*\alpha _{n}=1$, and $\beta _{1}^*\beta _{1}+\cdots +\beta _{n}^*\beta _{n}=1$. Such a rectangular convex combination is \textit{proper} if $\alpha _{i}, \beta _{i}$ are surjective, and \textit{trivial} if $\alpha _{i}^*\alpha _{i}=\lambda _{i}I$, $\beta _{i}^*\beta _{i}= \lambda _{i}I$ , and $\alpha _{i}^*\phi_{i}\beta _{i}=\lambda _{i}\phi $ for some $\lambda _{i}\in [0,1].$ A completely contractive linear map $\phi : X\rightarrow B(H,K)$ is a \textit{ rectanagular operator extreme point} if any proper rectangular convex combination of it is trivial.

\begin{prop}\cite[Proposition 1.12]{FHL18}\label{extrmpt}
Suppose that $\phi:X\rightarrow B(H,K)$ is a completely contarctive map and $\mathcal{S}(\phi):\mathcal{S}(X)\rightarrow B(K\oplus H)$ is the associated unital completely positive map defined on the Paulsen system. The following assertions are equivalent:
\begin{enumerate}
    \item $\mathcal{S}(\phi)$ is a pure completely positive map;\\
    \item $\mathcal{S}(\phi)$ is an operator extreme point;\\
    \item $\phi$ is a rectangular operator extreme point.
\end{enumerate}

\end{prop}

The following theorem characterizes the boundary representations of TRO's for operator spaces.

\begin{thm} Let $X$ be an operator space generating a TRO $T$. Let $\phi : T \rightarrow B(H,K)$ be an irreducible  representation of $T$. Then $\phi$ is a boundary representation for $X$ if and only if the following conditions are satisfied:
\begin{enumerate}
    \item[(i)] $\phi_{|_X}$ is a rectangular operator extreme point. \\
     \item[(ii)] $\phi$ is a finite representation for $X$.\\
      \item[(iii)] $X$ separates $\phi$.
\end{enumerate}
\end{thm}

\begin{proof}
Assume that $\phi : T\rightarrow B(H,K)$ is a boundary representation for $X$. Let $\omega : C^*(\mathcal{S}(X))\rightarrow B(K\oplus H)$ be a representation of $\mathcal{S}(X)$ such that $\omega _{|_{\mathcal{S}(X)}}=\mathcal{S}(\phi)$.  By Theorem \ref{opsbdy}, $\omega$ is a boundary representation for $S(X)$. Using \cite[Theorem 2.4.5]{Arv69}, we have $\mathcal{S}(\phi)$ is a pure UCP map, $\omega $ is a finite representation for $\mathcal{S}(X)$ and $\mathcal{S}(X)$ separates $\omega$. Thus, Proposition \ref{extrmpt} implies that $\phi_{|_X}$ is a rectangular operator extreme point, Proposition \ref{finite} implies that $\phi$ is a finite representation for $X$ and Proposition \ref{separates} implies that $X$ separates $\phi$.

Conversely, assume that all the three conditions are satisfied. Using   Proposition \ref{extrmpt}, Proposition \ref{converse finite} and Proposition \ref{converse separates} we have,
$\mathcal{S}(\phi)$ is a pure UCP map, $\omega $ is a finite representation for $\mathcal{S}(X)$ and $\mathcal{S}(X)$ separates $\omega$. Thus \cite[Theorem 2.4.5]{Arv69} implies that $\omega $ is a boundary representation for $\mathcal{S}(X)$. By \cite[Proposition 1.8]{FHL18},  $\phi$ is a boundary representation for $X$.
\end{proof}

\section{Rectangular hyperrigidity}
In this section, we introduce the notion of  rectangular hyperrigidity in the context of operator spaces in TRO's. Rectangular hyperrigidity is the generalization of Arveson's \cite{Arv11} notion of hyperrigidity in the context of  operator systems in $C^*$-algebras. We define rectangular hyperrigidity  as follows:

\begin{defn}\label{RHDF}
A finite or countably infinite set $G$ of generators of a TRO $T$ is said to be \textit{rectangular hyperrigid} if for every faithful representation from $T$ to  $B(H,K)$ and  every sequence of completely contractive (CC) maps $\phi _{n}: B(H,K)\rightarrow B(H,K)$ with $\Vert \phi _{n}\Vert _{cb}=1$, $n=1,2\cdots $,
\begin{equation}\label{RHEQ}
\lim \limits_{n\rightarrow \infty}\Vert \phi _{n}(g)-g \Vert =0, \hspace{2mm} \forall~~ g\in G \hspace{2mm} \implies \hspace{2mm} \lim \limits_{n\rightarrow \infty}\Vert \phi _{n}(t)-t \Vert =0, \hspace{2mm} \forall ~~t\in T.
\end{equation}
\end{defn}

As in Arveson's \cite{Arv11} notion of hyperrigity, we have lightened the notion of rectangular hyperrigidity by identifying $T$ with image $\pi(T)$ where  $\pi:T\rightarrow B(H,K)$ is  a faithful nondegenerate representation. Significantly, rectangular hyperrigid set of operators implies not only that equation \ref{RHEQ} should hold for sequences of CC maps $\phi_n$ with $\Vert \phi _{n}\Vert _{cb}=1$, but also that the property should persist for every other faithful representation of $T$.

\begin{prop}
Let $T$ be a TRO and $G$ a generating subset of $T$. Then $G$ is rectangular hyperrigid if and only if linear span of $G$ is rectangular hyperrigid.
\end{prop}
\begin{proof}
The proof follows directly from the definition of rectangular hyperrigidity.
\end{proof}

\begin{prop}
Let $A$ be a $C^*$-algebra and $S$ be an operator system in $A$ such that $A=C^*(S)$. If $S$ is rectangular hyperrigid, then $S$ is hyperrigid.
\end{prop}

\begin{proof}
The proof follows from the fact that (see \cite[Proposition 3.6]{Pau02}), every UCP map is completely bounded with CB norm 1.
\end{proof}

If   $T$ is a $C^*$-algebra and  $H=K$ in definition \ref{RHDF}, then by \cite[Proposition 2.11]{Pau02} and  \cite[Proposition 3.6]{Pau02} notions of rectangular hyperrigidity and hyperrigidity coincide. Thus, rectangular hyperrigidity is a generalized notion of hyperrigidity adapted in the context of TROs.

Now, we prove a characterization of rectangular hyperrigid operator spaces which leads to study the operator space analogue of Saskin's theorem (\cite{YAS66}, \cite[Theorem  4]{BL75}) relating retangular hyperrigity and boundary representations for operator spaces.

\begin{thm}\label{REHYP}

For every separable operator space $X$ that generates a TRO $T$, the following are equivalent:

\begin{enumerate}
      \item[(i)] $X$ is rectangular hyperrigid.

      \item[(ii)]  For every nondegenerate representation  $\pi: T\rightarrow B(H_{1},K_{1})$ on  seperable Hilbert spaces and every sequence $\phi _{n} : T \rightarrow B(H_{1},K_{1})$ of CC maps with $\Vert \phi _{n} \Vert _{cb}=1 $, $n=1,2,... $
      $$\lim \limits_{n\rightarrow \infty}\Vert \phi _{n}(x)-\pi(x) \Vert =0, \hspace{2mm} \forall~ x\in X \hspace{2mm} \implies \hspace{2mm} \lim \limits_{n\rightarrow \infty}\Vert \phi _{n}(t)-\pi(t) \Vert =0, \hspace{2mm} \forall ~t\in T.$$

       \item[(iii)] For every nondegenerate representation $\pi: T\rightarrow B(H_{1}, K_{1})$  on  seperable Hilbert spaces, $\pi _{|_{X}}$ has  unique extension property.

       \item[(iv)] For every TRO $T_1$, every triple morphism of TRO's $\theta : T\rightarrow T_{1}$ with $\Vert \theta \Vert _{cb}=1$ and every completely contractive map $\phi: T_{1}\rightarrow T_{1}$ with $\Vert \phi \Vert _{cb}=1$,
        $$\phi(x)=x, \text{ }\text{ }\forall~  x\in \theta(X) \implies \phi(t)=t,  \text{ }\text{ }\forall~ t\in \theta(T). $$
\end{enumerate}
\end{thm}

The spirit and the line of argument in the proof of the above theorem are the same as those by Arveson \cite[Theorem 2.1]{Arv11}, where we can replace operator systems, UCP maps and representations of $C^*$-algebras by operator spaces, CC maps and triple morphism of TROs. Further, we need to use Haagerup-Paulsen-Wittstock \cite[Theorem 8.2]{Pau02} extension theorem in place of Arveson extension theorem \cite[Theorem 1.2.3]{Arv69}.

\begin{ex}
Let $H$ be an infinite dimensional Hilbert space and $V$ be the unilateral right shift operator on $H$. Then the operator space $S=\text{span}\{I,V,V^*\}$ is not rectangular hyperrigid. To see this,  take $\phi_{n} : B(H)\rightarrow B(H)$  as $\phi_{n}=V^*I_{S}(\cdot)V$ for each $n=1,2\cdots $, where $I_{S} $ is the identity representation of $C^*(S)$. Then $\phi_{n}$ is a completely contractive linear map with $\Vert \phi_{n} \Vert _{cb} =1$ and $\phi_{n}$ is identity on $S$. Hence $\lim \limits_{n\rightarrow \infty } \Vert \phi_{n}(s)-s \Vert =0$ $\forall s\in S$ but $\lim \limits_{n\rightarrow \infty } \Vert \phi_{n}(VV^*)-VV^* \Vert = \Vert I-VV^* \Vert =1$. Note that the arguments in this example carries over to any isometry V which is not a unitary.
\end{ex}

We deduce the following necessary conditions for rectangular hyperrigidity:

\begin{cor}
Let $X$ be a separable operator space generating a TRO $T$. If $X$ is rectangular hyperrigid then every irreducible representation of $T$ is a boundary representation for $X$.
\end{cor}
\begin{proof}
The assertion is an immediate consequence of condition (ii) of Theorem \ref{REHYP}.
\end{proof}

\begin{problem}\label{RHP}
If every irreducible representation of TRO a $T$ is a boundary representation for a separable operator space $X\subseteq T$, then is $X$  rectangular hyperrigid ?
\end{problem}

\begin{prop}\label{dirsum}
Let $X$ be an operator space generating TRO $T$. Let $\pi_i:T\rightarrow B(H_i,K_i)$ be a non-degenerate representation such that $\pi_{i{|_X}}$ has  unique extension property for $i=1,2,...,n$. Then the direct sum of rectangular operator states
$$\oplus_{i=1}^n \pi_ {i{|_X}}:X\rightarrow B(\oplus_{i=1}^nH_i,\oplus_{i=1}^nK_i) $$
has  unique extension property.
\end{prop}

\begin{proof}
Assume that  $\pi_{i{|_X}}:T\rightarrow B(H_{i},K_{i})$ has unique extension property for $X$, $i=1,2,\cdots n$. By Corollary \ref{spacetopalsym},  $\mathcal{S}(\pi_{i{|_X}}): \mathcal{S}(X)\rightarrow B(K_{i}\oplus H_{i})$ has unique extension property for $\mathcal{S}(X)$. Using \cite[Proposition 4.4]{Arv11}, $\oplus _{i=1}^{n} \mathcal{S}(\pi_{i{|_X}})$ has unique extension property for $\mathcal{S}(X)$.
Note that $\oplus_{i=1}^{n} \mathcal{S}(\pi_{i{|_X}})=\mathcal{S}(\oplus _{i=1}^{n} \pi_{i{|_X}})$. Therefore,  by \cite[Proposition 1.8]{FHL18},   $\oplus_{i=1}^{n} \pi_{i{|_X}}$ has unique extension property for $X$.

\end{proof}

Here, we settle the Problem \ref{RHP} when TRO is finite dimensional. Thus,  we have a finite dimensional version of the classical Saskin's theorem in the context of operator spaces and TROs.

\begin{thm}
Let $X$ be an operator space whose generated TRO $T$ is finite dimensional, such that every irreducible representation of $T$ is a boundary representation for $X$. Then $X$ is rectangular hyperrigid.
\end{thm}
\begin{proof}
Using item (iii) of Theorem \ref{REHYP}, it is enough to prove that for every nondegenerate representation $\pi:T\rightarrow B(H,K)$, the rectangular operator state $\pi_{|_X}$ has  unique extension property. Since $T$ finite dimensional, \cite[Theorem 3.1.7]{DB11} implies that every nondegenerate representation of a finite dimensional TRO is a finite direct sum of irreducible repersentations. By our assumption every irreducible representation restricted to $X$ has unique extension property. By Proposition \ref{dirsum} finite direct sum of irreducible representation restricted to $X$ has unique extension property. Therefore every nondegenerate representation restricted to $X$ has  unique extension property.
\end{proof}

Now, we explore  relations between rectangular hyperrigity of an operator space and hyperrigidity of the corresponding Paulsen system.

\begin{thm}\label{HYIRHY}
Let $X$ be a separable operator space generating a TRO $T$. Paulsen system $\mathcal S(X)$ is  hyperrigid in the $C^*$-algebra $C^*(\mathcal S(X))$ if and only if $X$ is   rectangular hyperrigid in TRO $T$.
\end{thm}

\begin{proof}
Assume that  Paulsen system $\mathcal S(X)$ is  hyperrigid in the $C^*$-algebra $C^*(\mathcal S(X))$.\\
Let $\phi _{n}: B(H,K)\rightarrow B(H,K)$  be  CC maps with $\Vert \phi _{n}\Vert _{cb}=1$, $n=1,2\cdots $, such that

$$\lim \limits_{n\rightarrow \infty}\Vert \phi _{n}(x)-x \Vert =0 \hspace{2mm}, \forall~ x\in X.$$
Then the corresponding maps $\mathcal S(\phi _{n}):B(K\oplus H)\rightarrow B(K\oplus H)$, $n=1,2\cdots$ are UCP maps.
For all $ x,y\in X$ and  $ \lambda , \mu \in \mathbb{C}$, we have
\begin{eqnarray*}
 \left\Vert \mathcal S( \phi _{n})(\begin{bmatrix} \lambda & x \\ y^* & \mu\end{bmatrix})- \begin{bmatrix} \lambda & x \\ y^* & \mu \end{bmatrix} \right\Vert & =& \left\Vert \begin{bmatrix} \lambda & \phi _{n}(x) \\ \phi _{n}(y)^* & \mu  \end{bmatrix}- \begin{bmatrix} \lambda & x \\ y^* & \mu  \end{bmatrix} \right\Vert  \\
&=& \left\Vert \begin{bmatrix} 0 & \phi _{n}(x)-x \\ \phi _{n}(y)^*-y^* & 0 \end{bmatrix} \right\Vert \\
& \leq & \Vert \phi _{n}(x)-x \Vert + \Vert \phi_{n}(y)^* -y^* \Vert .
\end{eqnarray*}
Since $\mathcal S(X)$ is hyperrigid in $C^*(\mathcal S(X))$,  we conclude that for every $t\in T$
$$\lim\limits_{n\rightarrow \infty}\left\Vert  \mathcal S(\phi_{n})(\begin{bmatrix} 0 & t \\ 0 & 0\end{bmatrix})- \begin{bmatrix} 0 & t \\ 0 & 0\end{bmatrix} \right\Vert =0.$$
Thus,
$$\lim\limits_{n\rightarrow \infty}\Vert \phi_{n} (t)  - t \Vert =0 \text{  } \forall~ t\in T.$$

Conversely, suppose $X$ is rectangular hyperrigid.
By item (iii) of Theorem \ref{REHYP},  every nondegenerate representation of $T$ restricted to $X$ has  unique extension propery. From \cite[Proposition 3.1.2 and Equation 3.1]{BM04}, we have a one to one correspondence between the representations of a TRO $T$ and its linking algebra. Using Corollary \ref{spacetopalsym}, we get that every nondegenerate representation of the $C^*$-algebra $C^*(\mathcal S(X))$ restricted to $\mathcal S(X)$ has  unique extension property. Thus, by \cite[Theorem 2.1]{Arv11}, $\mathcal S(X)$ is hyperrigid.
\end{proof}

The following remark gives some partial answers to the Problem \ref{RHP} with extra assumptions.

\begin{rem}
Let $X$ be an operator space in $B(H,K)$ and  $T$ be a TRO  containing $X$ and generated by $X$. Assume that every irreducible representation $\psi:T\rightarrow B(H,K)$  is  a boundary representation for $X$. Then by Theorem \ref{opsbdy}, each $\mathcal S(\psi)$ is a boundary representation of  $C^*(\mathcal S (X))$ for $\mathcal S(X)$. If either $C^*(\mathcal S (X))$ has countable spectrum \cite[Theorem 5.1]{Arv11} or $C^*(\mathcal S (X))$ is a Type I $C^*$-algebra with $C^*(\mathcal S (X))''$ as the codomain for UCP maps on $C^*(\mathcal S (X))$ \cite[Corollary 3.3]{CK14}, then $\mathcal{S}(X)$ is hyperrigid. Therefore by Theorem \ref{HYIRHY}, $X$ is rectangular hyperrigid in $T$.
\end{rem}

% ------------------------------------------------------------------------

\subsection*{Acknowledgment}
The authors would like to thank Michael Hartz for some valuable comments and suggestions regarding this manuscript. The research  work of the first  author is supported by NBHM (National Board of Higher Mathematics, India) Ph.D Scholarship File No. 0203/17/2019/R\&D-II/10974. The research work of the second author is supported by NBHM (National Board of Higher Mathematics, India) Post Doctoral fellowship File No. 0204/26/2019/R\&D-II/12037. The  authors   like to thank Kerala School of Mathematics (KSoM), Kozhikode, Kerala, India for the   discussion meeting on 'Non-Commutative Convexity' held there during which work on this article was started.

% ------------------------------------------------------------------------
\end{document}